\documentclass[11pt,reqno]{amsart}

\usepackage[letterpaper,margin=1.15in]{geometry}
\usepackage{amsmath,amssymb,mathtools}
\usepackage{aliascnt}
\usepackage{microtype}
\usepackage[dvipsnames]{xcolor}
\usepackage[colorlinks=true,linkcolor=MidnightBlue,
  citecolor=ForestGreen,urlcolor=BrickRed]{hyperref}
\usepackage[nameinlink,capitalise,noabbrev]{cleveref}

\numberwithin{equation}{section}

\theoremstyle{plain}
\newtheorem{theorem}{Theorem}[section]
\newaliascnt{proposition}{theorem}
\newtheorem{proposition}[proposition]{Proposition}
\aliascntresetthe{proposition}
\newaliascnt{lemma}{theorem}
\newtheorem{lemma}[lemma]{Lemma}
\aliascntresetthe{lemma}
\newaliascnt{corollary}{theorem}
\newtheorem{corollary}[corollary]{Corollary}
\aliascntresetthe{corollary}

\theoremstyle{definition}
\newaliascnt{definition}{theorem}
\newtheorem{definition}[definition]{Definition}
\aliascntresetthe{definition}
\newaliascnt{example}{theorem}
\newtheorem{example}[example]{Example}
\aliascntresetthe{example}

\theoremstyle{remark}
\newaliascnt{remark}{theorem}
\newtheorem{remark}[remark]{Remark}
\aliascntresetthe{remark}

\newcommand{\C}{\mathbb C}
\newcommand{\Z}{\mathbb Z}
\newcommand{\D}{\mathbb D}
\newcommand{\Ds}{\mathbb D^*}
\newcommand{\eps}{\varepsilon}
\newcommand{\dd}{\,\mathrm d}
\newcommand{\Area}{\operatorname{Area}}
\newcommand{\ord}{\operatorname{ord}}

\title[Polynomial area growth for special K\"ahler metrics]
{Polynomial Area Growth and Cubic Singularities of\\
Affine Special K\"ahler Metrics in Real Dimension Two}

\author{Yiqian Shi}
\author{Ke Wang}
\author{Bin Xu}

\address[Yiqian Shi and Bin Xu]{School of Mathematical Sciences,
University of Science and Technology of China, Hefei 230026, China;
CAS Wu Wen-Tsun Key Laboratory of Mathematics, University of Science
and Technology of China, Hefei 230026, China}
\address[Ke Wang]{School of Mathematical Sciences, Luoyang Normal University, Luoyang 471934, China}
\email[Ke Wang, corresponding author]{wangke16@mails.ucas.edu.cn}
\date{August 2026}


\hypersetup{
  pdftitle={Polynomial Area Growth and Cubic Singularities of Affine Special Kahler Metrics in Real Dimension Two},
  pdfauthor={Yiqian Shi, Ke Wang and Bin Xu},
  pdfsubject={Polynomial area growth excludes essential singularities of the associated cubic differential},
  pdfkeywords={affine special Kahler metric, isolated singularity, cubic differential, polynomial area growth, Ahlfors-Schwarz lemma}
}

\begin{document}

\begin{abstract}
Let $g=w|dz|^2$ be an affine special K\"ahler metric on a punctured
disc and let $\Xi=q(z)\,dz^3$ be its associated holomorphic cubic
differential.  We prove that if the $g$-area of concentric punctured
neighbourhoods grows at most polynomially, then $q$ is meromorphic at
the puncture.  More precisely, an area bound of order $\eps^{-N}$ forces
every pole of $q$ to have order strictly less than $N+3$.  Hence an
essential singularity of the cubic differential rules out every
polynomial area upper bound.  The proof associates to $(g,\Xi)$ a
curvature $-1$ conformal pseudometric and applies the Ahlfors--Schwarz
lemma, followed by the submean inequality for holomorphic functions.
We also discuss the identically zero cubic differential and the
borderline logarithmic growth responsible for the strict pole-order
bound.
\end{abstract}

\maketitle

\section{Introduction}

Special K\"ahler structure lies at the interface 
of differential geometry and mathematical physics. It 
originated in the study of $N=2$ supersymmetric gauge theory. 
In mathematics, special K\"ahler structures arise naturally on 
the base of holomorphic Lagrangian fibrations and in the theory 
of Hitchin systems, and they have become an active research area 
in both geometry and physics.
	
An affine special K\"ahler structure is a K\"ahler structure together
with a flat, torsion-free, symplectic connection satisfying an additional
symmetry condition; see \cite{Freed1999}.
The departure of the connection $\nabla$ from the 
Levi--Civita connection is measured by a holomorphic
cubic differential $\Xi$. Singular special K\"ahler metrics arise
naturally on bases of holomorphic Lagrangian fibrations.  Lu proved that
every complete affine special K\"ahler metric is flat
\cite[Theorem~1]{Lu1999}.  This rigidity makes isolated singularities an
important part of the local theory.

In complex dimension one, Haydys \cite[Theorem~1.1]{Haydys2015} 
classified the asymptotic behaviour
of an affine special K\"ahler metric near an isolated puncture, assuming
that the associated cubic differential has no essential singularity.  If
\[
  g=w|dz|^2,\qquad \Xi=q(z)\,dz^3,
\]
and $n=\ord_0\Xi\in\Z$, then the metric has one of the 
following two asymptotic forms:
\begin{equation}\label{eq:Haydys-types}
  w=-|z|^{n+1}\log|z|\,e^{O(1)}
  \quad\text{or}\quad
  w=|z|^\beta\bigl(C+o(1)\bigr),
  \qquad C>0,\quad \beta<n+1.
\end{equation}
The hypothesis on $\Xi$ cannot simply be omitted from that classification:
special K\"ahler structures whose cubic differentials have essential
singularities do exist; see \cite{HaydysXu2020}.

Li and Xu \cite[Theorem~1.2]{LiXu2020} considered a different local
problem.  They classified isolated singularities of conformal metrics
with zero Gaussian curvature whose area satisfies a polynomial growth
condition.  In the notation used here, their condition is that there exist
\begin{equation}\label{eq:area-intro}
  0<R<1,\qquad M>0,\qquad N\geq0
\end{equation}
such that
\begin{equation}\label{eq:area-r}
  \Area_g\{1/r<|z|<R\}\leq Mr^N
  \qquad\text{for every }r>1/R.
\end{equation}
The radii in \eqref{eq:area-r} are Euclidean coordinate radii, not
geodesic radii.  The property is unchanged by a holomorphic coordinate
change fixing the puncture \cite[Remark~1.5]{LiXu2020}.

One cannot directly apply the theorem of Li and Xu to a special K\"ahler
metric: in special K\"ahler geometry it is the connection that is flat,
whereas the Riemannian metric generally has nonzero Gaussian curvature.  The
purpose of this note is to show that the \emph{area condition} from
\cite{LiXu2020}, when combined with the special K\"ahler equation, does
exclude essential cubic singularities.

For $0<\eps<R$, set
\begin{equation}\label{eq:area-def}
  A_g(\eps;R)
  :=\Area_g\{\eps<|z|<R\}
  =\int_{\eps<|z|<R} w\,\dd x\,\dd y.
\end{equation}
Thus \eqref{eq:area-r} is equivalent to
\begin{equation}\label{eq:area-eps}
  A_g(\eps;R)\leq M\eps^{-N},
  \qquad 0<\eps<R.
\end{equation}

\begin{definition}[Polynomial area growth]\label{def:area-growth}
We say that a conformal metric $g=w|dz|^2$ on $\Ds$ has
\emph{polynomial area growth at the puncture} if there exist
$R\in(0,1)$, $M>0$, and $N\geq0$ for which
\eqref{eq:area-eps} holds.
\end{definition}

Our main result is the following.

If $\Xi$ is nonzero and meromorphic at the origin, we write
$\ord_0\Xi:=\ord_0q$.  This integer does not depend on the chosen
coordinate.  Indeed, if $\zeta=\phi(z)$ is such a change and
$\Xi=\widetilde q(\zeta)\,d\zeta^3$, then
$\widetilde q(\phi(z))=q(z)\phi'(z)^{-3}$, and $\phi'$ does not vanish
near the origin.  Moreover, $\phi(z)=z\psi(z)$ for a holomorphic function
$\psi$ with $\psi(0)\neq0$.  Hence the orders computed in the $z$- and
$\zeta$-coordinates agree.

\begin{theorem}\label{thm:main}
Let $(g,I,\omega,\nabla)$ be an affine special K\"ahler structure on
$\Ds=\{0<|z|<1\}$, and write
\[
  g=w|dz|^2,
  \qquad
  \Xi=q(z)\,dz^3.
\]
Assume that $g$ has polynomial area growth in the sense of
\Cref{def:area-growth}, with exponent $N$.  Then $q$ extends
meromorphically across $0$.
Consequently, either $q\equiv0$, or $\ord_0\Xi$ is a finite integer.
Moreover, if $q$ has a pole of order $p$ at $0$, then
\begin{equation}\label{eq:pole-bound}
  p<N+3.
\end{equation}
\end{theorem}

The theorem has two immediate interpretations.  First, the contrapositive
says that an essential singularity of $\Xi$ prevents the area from being
$O(\eps^{-N})$ for every fixed $N$.  Second, except for the harmless case
$\Xi\equiv0$, the polynomial area hypothesis may replace the finite-order
hypothesis in \cite{Haydys2015}; the asymptotic alternatives
\eqref{eq:Haydys-types} then follow.

\paragraph{Proof mechanism.}
The proof uses only three standard ingredients: the local special K\"ahler
equation, an elementary maximum-principle proof of the Ahlfors--Schwarz
lemma, and the submean inequality for the absolute value of a holomorphic
function.  The only point at which special K\"ahler geometry enters is the
local equation recorded in \Cref{prop:SK-equation}.

The paper is organized as follows.  \Cref{sec:local-equation} recalls the
local special K\"ahler equation and establishes the
Ahlfors--Schwarz estimate used below. \Cref{sec:main-proof} proves the main theorem
using the associated hyperbolic pseudometric, and gives an alternative proof
 using Laurent coefficients.  \Cref{sec:consequences}
discusses essential singularities, area lower bounds, the zero cubic
differential, and the equivalence with the finite-order hypothesis.

\section{The local special K\"ahler equation and the Ahlfors–Schwarz estimate}
\label{sec:local-equation}

To relate the area growth of the special K\"ahler metric to the
singularity of its associated cubic differential, we first recall
the local special K\"ahler equation. We begin by fixing our normalisations.
The Euclidean Laplacian is
\[
  \Delta=\partial_x^2+\partial_y^2,
  \qquad z=x+iy.
\]

\begin{definition}
An affine special K\"ahler structure on a Riemann surface is a quadruple
$(g,I,\omega,\nabla)$ such that $(g,I,\omega)$ is K\"ahler,
$\nabla$ is flat, torsion-free and symplectic, and
\begin{equation}\label{eq:SK-def}
  (\nabla_XI)Y=(\nabla_YI)X
\end{equation}
for all vector fields $X,Y$.
\end{definition}

The difference between $\nabla$ and the Levi--Civita connection determines
a holomorphic cubic differential $\Xi$.  In a holomorphic coordinate it
has the form $\Xi=q(z)\,dz^3$.  The following equation is standard; we
spell out how it follows from the local formulae in
\cite{Haydys2015,HaydysXu2020}.

\begin{proposition}[Local special K\"ahler equation]
\label{prop:SK-equation}
Let $g=e^{-u}|dz|^2$ be an affine special K\"ahler metric on a punctured
disc and let $\Xi=q(z)\,dz^3$.  Then
\begin{equation}\label{eq:SK-PDE}
  \Delta u=16|q|^2e^{2u}.
\end{equation}
Equivalently, if $w=e^{-u}$, then
\begin{equation}\label{eq:logw-PDE}
  -\Delta\log w=16\frac{|q|^2}{w^2}.
\end{equation}
\end{proposition}

\begin{proof}
We review the local-coordinate calculation.  On a simply connected domain,
the special K\"ahler equations can be written in terms of a harmonic
function $h$ as
\begin{equation}\label{eq:KW-simple}
  \Delta h=0,
  \qquad
  \Delta u=|\dd h|^2e^{2u}.
\end{equation}
On a punctured disc there is one additional cohomology parameter.  If
\begin{equation}\label{eq:phi}
  \varphi=\frac{y\,\dd x-x\,\dd y}{x^2+y^2},
\end{equation}
then every special K\"ahler structure determines a harmonic function $h$
and a real number $a$ such that
\begin{equation}\label{eq:KW-punctured}
  \Delta h=0,
  \qquad
  \Delta u=|\dd h+a\varphi|^2e^{2u}.
\end{equation}
These are the local formulae of
\cite[Corollaries~2.2 and~2.3]{Haydys2015}; their derivation is obtained by
writing the flat symplectic connection in a real frame and imposing
\eqref{eq:SK-def}.

With the same normalisation, the coefficient of the associated cubic
differential is
\begin{equation}\label{eq:q-ha}
  q=\frac12\left(\frac{a}{2z}-i\frac{\partial h}{\partial z}\right).
\end{equation}
Writing $\partial h/\partial z=(h_x-ih_y)/2$ and substituting
\eqref{eq:phi}, a direct calculation gives
\begin{equation}\label{eq:q-norm}
  16|q|^2=|\dd h+a\varphi|^2,
\end{equation}
where the norm on the right is Euclidean.  Substitution of
\eqref{eq:q-norm} into \eqref{eq:KW-punctured} proves
\eqref{eq:SK-PDE}.  Finally, $u=-\log w$ gives
\eqref{eq:logw-PDE}.
\end{proof}

\begin{remark}\label{rem:curvature}
The Gaussian curvature of $g=w|dz|^2$ is
\[
  K_g=-\frac{1}{2w}\Delta\log w
      =8\frac{|q|^2}{w^3}\geq0.
\]
Thus the Riemannian metric is generally not flat.  This is the precise
reason why the flat-metric classification of \cite{LiXu2020} cannot be
applied directly.
\end{remark}

Having recorded the local equation (2.2), we now introduce 
a hyperbolic conformal pseudometric, which will bridge it to the area hypothesis. We use the classical
Ahlfors--Schwarz comparison \cite{Ahlfors1938}, but include the short
maximum-principle proof needed here so that zeros of the pseudometric cause
no ambiguity.

\begin{definition}
A conformal pseudometric on a plane domain $\Omega$ is an expression
$\lambda^2|dz|^2$, where $\lambda\colon\Omega\to[0,\infty)$ is continuous
and is positive and $C^2$ away from isolated zeros.  At points where
$\lambda>0$, its Gaussian curvature is
\begin{equation}\label{eq:conf-curv}
  K_\lambda=-\frac{1}{\lambda^2}\Delta\log\lambda.
\end{equation}
\end{definition}

We only need pseudometrics whose zeros have the form
\begin{equation}\label{eq:pseudometric-zero}
  \lambda(z)=|z-z_0|^m e^{v(z)},
  \qquad m\in\Z_{\geq1},
\end{equation}
with $v$ smooth.  In particular, $\log\lambda$ tends to $-\infty$ at a
zero.

\begin{lemma}[Ahlfors--Schwarz with isolated zeros]
\label{lem:AS-disc}
Let $\lambda^2|d\zeta|^2$ be a conformal pseudometric on the unit disc
$\D$.  Suppose its zeros satisfy \eqref{eq:pseudometric-zero} and
\begin{equation}\label{eq:AS-hyp}
  \Delta\log\lambda\geq\lambda^2
\end{equation}
wherever $\lambda>0$.  Then
\begin{equation}\label{eq:AS-disc-bound}
  \lambda(\zeta)\leq\frac{2}{1-|\zeta|^2}.
\end{equation}
\end{lemma}

\begin{proof}
Fix $s\in(0,1)$ and consider the Poincar\'e density on the disc
$\D_s=\{|\zeta|<s\}$,
\begin{equation}\label{eq:poincare-s}
  \lambda_s(\zeta)=\frac{2s}{s^2-|\zeta|^2}.
\end{equation}
It satisfies
\begin{equation}\label{eq:poincare-s-eq}
  \Delta\log\lambda_s=\lambda_s^2
\end{equation}
and tends to $+\infty$ as $|\zeta|\to s$.

The quotient $\lambda/\lambda_s$ is continuous on $\D_s$.  It tends to
zero at the boundary because $\overline{\D_s}$ is compactly contained in
$\D$, so $\lambda$ is bounded there, whereas $\lambda_s$ diverges.  It
also vanishes at every
zero of $\lambda$.  Suppose, for contradiction, that
$\lambda/\lambda_s$ takes a value greater than $1$.  It then has an
interior maximum at a point $\zeta_0$ where $\lambda(\zeta_0)>0$ and
\begin{equation}\label{eq:ratio-big}
  \lambda(\zeta_0)>\lambda_s(\zeta_0).
\end{equation}
Since the logarithm is strictly increasing,
$\log(\lambda/\lambda_s)$ also has a local maximum at $\zeta_0$.
This function is $C^2$ near $\zeta_0$, so
\[
\Delta\log\frac{\lambda}{\lambda_s}(\zeta_0)\leq 0.
\]
On the other hand, \eqref{eq:AS-hyp}, \eqref{eq:poincare-s-eq}, and
\eqref{eq:ratio-big} give
\[
  \Delta\log\frac{\lambda}{\lambda_s}(\zeta_0)
  \geq\lambda(\zeta_0)^2-\lambda_s(\zeta_0)^2>0,
\]
a contradiction.  Hence $\lambda\leq\lambda_s$ on $\D_s$.  For fixed
$\zeta\in\D$, letting $s\uparrow1$ proves \eqref{eq:AS-disc-bound}.
\end{proof}

\begin{corollary}[The punctured-disc estimate]
\label{cor:AS-punctured}
Let $\lambda^2|dz|^2$ be a conformal pseudometric on $\Ds$ satisfying the
hypotheses of \Cref{lem:AS-disc}.  Then
\begin{equation}\label{eq:cusp-bound}
  \lambda(z)\leq
  \frac{1}{|z|\log(1/|z|)}.
\end{equation}
\end{corollary}

\begin{proof}
The left half-plane
\[
  H_-:=\{\zeta\in\C:\operatorname{Re}\zeta<0\}
\]
is a universal cover of $\Ds$ by the map $\pi(\zeta)=e^\zeta$.  Pulling
back the pseudometric gives the density
\begin{equation}\label{eq:lift-density}
  \widehat\lambda(\zeta)
  =\lambda(e^\zeta)|e^\zeta|.
\end{equation}
The maximum-principle argument in \Cref{lem:AS-disc}, transported by a
M\"obius map from $H_-$ to $\D$, compares
$\widehat\lambda$ with the curvature $-1$ Poincar\'e density of $H_-$:
\begin{equation}\label{eq:left-density}
  \widehat\lambda(\zeta)\leq
  \frac{1}{-\operatorname{Re}\zeta}.
\end{equation}
Since $|e^\zeta|=e^{\operatorname{Re}\zeta}$ and
$-\operatorname{Re}\zeta=\log(1/|e^\zeta|)$, equations
\eqref{eq:lift-density}--\eqref{eq:left-density} give
\eqref{eq:cusp-bound}.
\end{proof}

\section{Polynomial area growth and meromorphicity of the cubic differential}
\label{sec:main-proof}

We now combine the local equation and the Ahlfors–Schwarz
 estimate from \Cref{sec:local-equation} to prove the main theorem. The proof is separated into
short steps in order to make every estimate explicit.

\begin{proof}[Proof of \Cref{thm:main}]
Fix $R$, $M$, and $N$ as in \Cref{def:area-growth}.

\smallskip
\noindent\emph{Step 1: Construct the hyperbolic pseudometric.}

Write $g=e^{-u}|dz|^2$ and $\Xi=q(z)\,dz^3$.  If $q\equiv0$, then it
already extends holomorphically across the puncture, so assume for now that
$q\not\equiv0$.  Define
\begin{equation}\label{eq:lambda-def}
  \lambda=4|q|e^u=\frac{4|q|}{w}.
\end{equation}
On the complement of the discrete zero set of $q$, holomorphicity gives
$\Delta\log|q|=0$.  Hence \Cref{prop:SK-equation} yields
\begin{equation}\label{eq:lambda-hyp}
  \Delta\log\lambda
  =\Delta u
  =16|q|^2e^{2u}
  =\lambda^2.
\end{equation}
If $q$ has a zero of order $m$ at $z_0\neq0$, then
\[
  q(z)=(z-z_0)^m h(z),
  \qquad h(z_0)\neq0,
\]
and consequently
\[
  \lambda(z)=|z-z_0|^m e^{v(z)}
\]
for a smooth function $v$.  Thus $\lambda^2|dz|^2$ satisfies all the
hypotheses of \Cref{cor:AS-punctured}.  We conclude that
\begin{equation}\label{eq:lambda-final-bound}
  \frac{4|q(z)|}{w(z)}
  \leq\frac{1}{|z|\log(1/|z|)}.
\end{equation}
Equivalently, 
\begin{equation}\label{eq:key-lower}
  |q(z)|\leqslant\frac{w(z)}{4|z|\log\frac1{|z|}}.
\end{equation}
This is the key estimate.  Notice its direction: the special K\"ahler
metric must become large wherever the cubic differential becomes large.

\medskip
\noindent\emph{Step 2: Convert the area bound into a pointwise bound.}

Fix $z\in\Ds$ with $|z|=\rho$, where $\rho>0$ is so small that
\begin{equation}\label{eq:rho-small}
  \frac{3\rho}{2}<R,
  \qquad
  \log(1/\rho)>2.
\end{equation}
Let
\begin{equation}\label{eq:small-disc}
  B_z:=B\left(z,\frac{\rho}{4}\right).
\end{equation}
For every $\zeta\in B_z$,
\begin{equation}\label{eq:radii-compare}
  \frac{3\rho}{4}<|\zeta|<\frac{5\rho}{4}.
\end{equation}
After decreasing the upper bound on $\rho$ once more, there is an absolute
constant $c_0>0$ such that
\begin{equation}\label{eq:log-compare}
  |\zeta|\log\frac1{|\zeta|}
  > c_0\rho\log\frac1\rho
  \qquad(\zeta\in B_z).
\end{equation}
Indeed, \eqref{eq:radii-compare} gives 
$|\zeta| > 3\rho/4$, while
\[
  \log(1/|\zeta|)  
  > \log(1/\rho)-\log(5/4)
  > \tfrac12\log(1/\rho)
\]
for all sufficiently small $\rho$.  Thus one may take, for example,
$c_0=1/4$.

Because $q$ is holomorphic on $B_z$, the function $|q|$ is subharmonic.
The submean inequality, followed by \eqref{eq:key-lower} and
\eqref{eq:log-compare}, gives
\begin{align}
  |q(z)|
  &\leq
  \frac{1}{\pi(\rho/4)^2}
  \int_{B_z}|q(\zeta)|\,\dd A_0(\zeta)
  \notag\\
  &<
  \frac{C_1}{\rho^3\log(1/\rho)}
  \int_{B_z}w(\zeta)\,\dd A_0(\zeta),
  \label{eq:submean-estimate}
\end{align}
where $\dd A_0=\dd x\,\dd y$ and $C_1$ is absolute.  By
\eqref{eq:radii-compare},
\[
  B_z\subset\{\rho/2<|\zeta|<R\}.
\]
The area hypothesis \eqref{eq:area-eps}, applied with $\eps=\rho/2$,
therefore implies
\begin{equation}\label{eq:q-growth}
  |q(z)|
  <
  \frac{C_2}{\log(1/\rho)}\rho^{-(N+3)},
  \qquad |z|=\rho,
\end{equation}
where $C_2=C_1M2^N$.  This bound is uniform in the argument of $z$.

\medskip
\noindent\emph{Step 3: Obtain meromorphic extension and bound the pole
order.}

Choose an integer $m>N+3$.  From \eqref{eq:q-growth},
\begin{equation}\label{eq:z-m-q}
  |z^mq(z)|
  <
  C_2\frac{|z|^{m-N-3}}{\log(1/|z|)}
  \longrightarrow0
  \qquad(z\to0).
\end{equation}
The function $z^mq(z)$ is holomorphic on $\Ds$ and bounded near $0$.
By Riemann's removable singularity theorem, it therefore 
extends holomorphically across $0$.  It follows that $q$ has at worst a 
pole of finite order at $0$.

It remains to prove the sharper bound \eqref{eq:pole-bound}.  Suppose that
$q$ has a pole of order $p$.  Then
\begin{equation}\label{eq:q-leading}
  q(z)=az^{-p}(1+o(1)),
  \qquad a\neq0.
\end{equation}
For all sufficiently small $\rho$,
\[
  |q(z)|\geq\frac{|a|}{2}\rho^{-p}
  \qquad(|z|=\rho).
\]
Combining this with \eqref{eq:q-growth} gives
\begin{equation}\label{eq:pole-compare}
  \frac{|a|}{2}
  <
  C_2\frac{\rho^{p-N-3}}{\log(1/\rho)}.
\end{equation}
If $p>N+3$, the right-hand side tends to zero.  It also tends to zero if
$p=N+3$, because of the logarithm in the denominator.  Both cases are
impossible, proving $p<N+3$ and completing the proof of
\Cref{thm:main}.
\end{proof}

For readers who prefer to see the exclusion of an essential singularity directly in terms of Laurent coefficients, we give an alternative proof of the pole-order bound in Step 3. It also makes the strict inequality in \eqref{eq:pole-bound} transparent.

\begin{proof}[Alternative proof of the meromorphic extension and the pole
bound]

Let
\begin{equation}\label{eq:Laurent}
  q(z)=\sum_{j=-\infty}^{\infty}a_jz^j
\end{equation}
be the Laurent expansion of $q$ on $\Ds$.  For every integer $k\geq1$
and every $0<t<R$, Cauchy's coefficient formula gives
\begin{align}
  |a_{-k}|
  &\leq
  \frac{t^k}{2\pi}
  \int_0^{2\pi}|q(te^{i\theta})|\,\dd\theta,
  \notag\\
  \int_0^{2\pi}|q(te^{i\theta})|\,\dd\theta
  &\geq2\pi|a_{-k}|t^{-k}.
  \label{eq:Cauchy-L1}
\end{align}
Choose a fixed $R_0\in(0,R)$ and take $0<\eps<R_0$.  Integrating the key estimate
\eqref{eq:key-lower} in polar coordinates and using
\eqref{eq:Cauchy-L1}, we obtain
\begin{align}
  A_g(\eps;R)
  &\geq
  4\int_\eps^{R_0}
  t^2\log(1/t)
  \left(\int_0^{2\pi}|q(te^{i\theta})|\,\dd\theta\right)\dd t
  \notag\\
  &\geq
  8\pi|a_{-k}|
  \int_\eps^{R_0}t^{2-k}\log(1/t)\,\dd t.
  \label{eq:Laurent-area-lower}
\end{align}

If $k>3$, then for sufficiently small $\eps$ there is a constant
$c_k>0$, independent of $\eps$, such that
\begin{equation}\label{eq:integral-k-large}
  \int_\eps^{R_0}t^{2-k}\log(1/t)\,\dd t
  \geq c_k\eps^{3-k}\log(1/\eps).
\end{equation}
For example, if $2\varepsilon<R_0$ and 
$\varepsilon\leqslant 1/4$ (in which case 
$\log\frac{1}{2\varepsilon}\geqslant\frac{1}{2}\log\frac{1}{\varepsilon}$), 
then on the interval $[\varepsilon, 2\varepsilon]$, $t^{2-k}\log(1/t)$ 
is monotonically decreasing; therefore
\[
	\int_{\varepsilon}^{R_0}t^{2-k}\log(1/t)dt \geqslant 
	\int_{\varepsilon}^{2\varepsilon}t^{2-k}\log(1/t)dt
	\geqslant\int_{\varepsilon}^{2\varepsilon}(2\varepsilon)^{2-k}\log(1/2\varepsilon)dt
	\geqslant 2^{1-k}\varepsilon^{3-k}\log(1/\varepsilon).
\]
	In this case \(c_k=2^{1-k}\).
For $k=3$ the integral is explicit:
\begin{equation}\label{eq:integral-k3}
  \int_\eps^{R_0}\frac{\log(1/t)}{t}\,\dd t
  =\frac12\left(\log^2(1/\eps)-\log^2(1/R_0)\right).
\end{equation}
Comparing \eqref{eq:Laurent-area-lower} with
$A_g(\eps;R)\leq M\eps^{-N}$ shows that
\begin{equation}\label{eq:coeff-vanish}
  a_{-k}=0
  \qquad\text{whenever }k-3\geq N.
\end{equation}
Indeed, if $k-3>N$, the power of $1/\eps$ on the right-hand side of
\eqref{eq:Laurent-area-lower} is too large; if $k-3=N$, the additional
factor $\log(1/\eps)$ still gives a contradiction.  For $N=0$ and $k=3$,
the square logarithm in \eqref{eq:integral-k3} gives the same conclusion.

Only finitely many negative Laurent coefficients remain.  Hence $q$ is
meromorphic, and the largest possible pole order is strictly less than
$N+3$. This gives an alternative proof of the conclusion of \Cref{thm:main} directly from Laurent's theorem.
\end{proof}

\section{Consequences and examples}
\label{sec:consequences}

\subsection{Essential cubic singularities admit no polynomial area bound}

\begin{corollary}\label{cor:essential}
If the associated cubic differential of an affine special K\"ahler metric
has an essential singularity at the origin, then for every choice of
$R\in(0,1)$, $M>0$, and $N\geq0$, there are arbitrarily small $\eps>0$
such that
\[
  A_g(\eps;R)>M\eps^{-N}.
\]
\end{corollary}

\begin{proof}
If the conclusion failed, then for some $R,M,N$ the estimate
$A_g(\eps;R)\leq M\eps^{-N}$ would hold for all sufficiently small
$\eps$.  Increasing $M$ if necessary makes the same estimate valid for
all $0<\eps<R$, because $g$ has finite area on every compact subannulus.
\Cref{thm:main} would then make $q$ meromorphic at the origin, a
contradiction.
\end{proof}

Thus the examples with an essential cubic differential constructed in
\cite{HaydysXu2020} cannot satisfy any polynomial area upper bound at
such a puncture.  This is an area statement about the special K\"ahler
metric $g$; no area condition on the associated hyperbolic pseudometric
is assumed.

\subsection{A quantitative lower bound from the pole order}

Suppose $q$ has a pole of order $p$, so that
$q(z)=az^{-p}(1+o(1))$ with $a\neq0$.  The key estimate
\eqref{eq:key-lower} yields, for sufficiently small $|z|$,
\begin{equation}\label{eq:w-pole-lower}
  w(z)\geq c|z|^{1-p}\log(1/|z|)
\end{equation}
for some $c>0$.  Integrating gives the following useful lower bounds:
\begin{equation}\label{eq:area-pole-lower}
  A_g(\eps;R_0)\geq
  \begin{cases}
    c_p\eps^{3-p}\log(1/\eps),&p>3,\\[2mm]
    c_3\log^2(1/\eps),&p=3,
  \end{cases}
\end{equation}
for a sufficiently small fixed $R_0$ and positive constants $c_p$.
This explains both the shift by $3$ in \eqref{eq:pole-bound} and why the
inequality is strict.  When $N=0$, for example, the cubic differential
can have at most a double pole.

The logarithmic asymptotic model in \eqref{eq:Haydys-types}, with
$n=-p$, has
\[
  w\asymp |z|^{1-p}\log(1/|z|),
\]
so the powers in \eqref{eq:area-pole-lower} are optimal up to the
borderline logarithmic factor.

\subsection{The identically zero cubic differential}

The phrase ``$\Xi$ has finite order'' normally excludes the zero
differential, whose order is either left undefined or declared to be
$+\infty$.  This is the only terminological exception to the statement
that polynomial area growth forces finite order.

\begin{example}\label{ex:zero-cubic}
On $\Ds$, consider
\begin{equation}\label{eq:cylinder}
  g=|z|^{-2}|dz|^2
\end{equation}
and take $\nabla$ to be the Levi--Civita connection.  The metric is flat,
so $\nabla$ is flat, torsion-free and symplectic, and $\nabla I=0$.
It is therefore an affine special K\"ahler structure with
$\Xi\equiv0$.  The puncture is a genuine cylindrical end, and
\begin{equation}\label{eq:cylinder-area}
  \Area_g\{1/r<|z|<R\}=2\pi\log(Rr),
\end{equation}
which satisfies a polynomial bound, for instance with $N=1$.
Nevertheless, $\ord_0\Xi$ is not a finite integer.
\end{example}

There is no analytic pathology here: $\Xi\equiv0$ extends holomorphically
and has no essential singularity.  Moreover, by
\Cref{rem:curvature}, $\Xi\equiv0$ implies that $g$ is flat.
In this special case the theorem of Li and Xu \cite{LiXu2020} applies
directly and supplies the appropriate local metric classification.

\subsection{Equivalence with the finite-order hypothesis}

Combining \Cref{thm:main} with \cite[Theorem~1.1]{Haydys2015} gives a
useful equivalent formulation.

\begin{corollary}\label{cor:Haydys}
Let $g=w|dz|^2$ be an affine special K\"ahler metric on $\Ds$ whose
associated cubic differential $\Xi$ is not identically zero.  The
following conditions are equivalent:
\begin{enumerate}
  \item $g$ has polynomial area growth at the puncture;
  \item $\Xi$ extends meromorphically across the puncture.
\end{enumerate}
If these conditions hold and $n=\ord_0\Xi$, then $w$ has one of the two
asymptotic forms in \eqref{eq:Haydys-types}.
\end{corollary}

\begin{proof}
The implication from (1) to (2) is \Cref{thm:main}.  Assume (2).
Because $\Xi$ is not identically zero, its order $n$ is a finite integer.
Haydys' classification \cite[Theorem~1.1]{Haydys2015} gives the two
alternatives in \eqref{eq:Haydys-types}.

It remains to verify (1).  In the logarithmic alternative,
$e^{O(1)}$ is bounded above near the origin, and hence the area is bounded
by a constant multiple of
\[
  \int_\eps^{R_0}r^{n+2}\log(1/r)\,\dd r.
\]
This integral stays bounded when $n>-3$, grows like
$\log^2(1/\eps)$ when $n=-3$, and has at most power growth when
$n<-3$.  In the power-law alternative, the area is bounded by a constant
multiple of
\[
  \int_\eps^{R_0}r^{\beta+1}\,\dd r.
\]
This integral stays bounded when $\beta>-2$, grows like
$\log(1/\eps)$ when $\beta=-2$, and has power growth when $\beta<-2$.
In every case some estimate of the form \eqref{eq:area-eps} holds.
\end{proof}

\begin{remark}
The exponent $N$ in the area hypothesis is not intrinsic.  If
\eqref{eq:area-eps} holds for $N$, it also holds for every larger
exponent.  The set of admissible exponents need not have a smallest
element: logarithmic area growth, for example, admits every $N>0$ but
not $N=0$.  The quantitative statement $p<N+3$ is therefore most
informative when $N$ is close to the infimum of the admissible exponents.
\end{remark}

\section*{Acknowledgements}

OpenAI's GPT-5.6 Sol and GPT-6 Astra assisted throughout the research and
preparation of this article, including proof development and manuscript
drafting.  The authors revised all AI-assisted material with careful
scrutiny and take full responsibility for all mathematical claims and for
the final text.  Y.S. is supported in part by the National Natural Science
Foundation of China (Grant No. 11931009).  B.X. is supported in part by the
Project of Stable Support for Youth Team in Basic Research Field, CAS
(Grant No. YSBR-001) and NSFC (Grant Nos. 12271495).


\begin{thebibliography}{99}

\bibitem{Ahlfors1938}
L.~V. Ahlfors,
\emph{An extension of Schwarz's lemma},
Transactions of the American Mathematical Society \textbf{43} (1938),
no.~3, 359--364,
\href{https://doi.org/10.1090/S0002-9947-1938-1501949-6}
{doi:10.1090/S0002-9947-1938-1501949-6}.

\bibitem{Freed1999}
D.~S. Freed,
\emph{Special K\"ahler manifolds},
Communications in Mathematical Physics \textbf{203} (1999), no.~1,
31--52,
\href{https://doi.org/10.1007/s002200050604}
{doi:10.1007/s002200050604}.

\bibitem{Haydys2015}
A. Haydys,
\emph{Isolated singularities of affine special K\"ahler metrics in two
dimensions},
Communications in Mathematical Physics \textbf{340} (2015), no.~3,
1231--1237,
\href{https://doi.org/10.1007/s00220-015-2441-6}
{doi:10.1007/s00220-015-2441-6}.

\bibitem{HaydysXu2020}
A. Haydys and B. Xu,
\emph{Special K\"ahler structures, cubic differentials and hyperbolic
metrics},
Selecta Mathematica (New Series) \textbf{26} (2020), Paper No.~37,
\href{https://doi.org/10.1007/s00029-020-00560-y}
{doi:10.1007/s00029-020-00560-y}.

\bibitem{LiXu2020}
J. Li and B. Xu,
\emph{Isolated singularities of flat metrics on Riemann surfaces},
Proceedings of the American Mathematical Society \textbf{148} (2020),
no.~9, 4057--4064,
\href{https://doi.org/10.1090/proc/15044}{doi:10.1090/proc/15044}.

\bibitem{Lu1999}
Z. Lu,
\emph{A note on special K\"ahler manifolds},
Mathematische Annalen \textbf{313} (1999), no.~4, 711--713,
\href{https://doi.org/10.1007/s002080050278}
{doi:10.1007/s002080050278}.

\end{thebibliography}
\end{document}